\documentclass{amsart}
\usepackage[T1]{fontenc}   
\usepackage{graphicx}
\usepackage{tikz}
\newcounter{qcounter}
\usepackage{amssymb}
\usepackage{url}
\usepackage{amscd}
\usepackage{graphicx}
\usepackage{tikz}
\usepackage{amssymb}
\usepackage{amsmath}
\usepackage{amsfonts, setspace}
\usepackage{amssymb,amsmath}
\usepackage{amsthm,amsmath,latexsym,amsfonts,mathrsfs,graphics,graphicx,amssymb}
\usepackage[mathcal]{eucal}

\newtheorem{theorem}{Theorem}[section]
\newtheorem{lemma}[theorem]{Lemma}
\newtheorem{proposition}[theorem]{Proposition}
\newtheorem{corollary}[theorem]{Corollary}
\theoremstyle{definition}

\numberwithin{equation}{section}

\begin{document}

\vspace{0.5in}

\renewcommand{\bf}{\bfseries}
\renewcommand{\sc}{\scshape}
\vspace{0.5in}

\title{Equivalence of $\mathbb{Z}_{4}$-actions on handlebodies of genus $g$}

\author{Jesse Prince-Lubawy}
\address{Department of Mathematics\\
University of North Alabama\\
Florence, AL 35632}
\email{jprincelubawy@una.edu}


\subjclass[2010]{Primary 57M60}

\keywords{handlebodies, orbifolds, graph of groups, orientation-preserving, cyclic actions}

\begin{abstract} 
In this paper we consider all orientation-preserving $\mathbb{Z}_{4}$-actions on $3$-dimensional handlebodies $V_g$ of genus $g>0$. We study the graph of groups $(\Gamma($v$),\mathbf{G(v)})$, which determines a handlebody orbifold $V(\Gamma($v$),{\mathbf{G(v)}})\simeq{V_g}/\mathbb{Z}_{4}$. This algebraic characterization is used to enumerate the total number of $\mathbb{Z}_{4}$ group actions on such handlebodies, up to equivalence.
\end{abstract}

\maketitle

\section{\bf Introduction}

A {$\mathbf{G}$-action} on  a handlebody $V_g$, of genus $g>0$, is a group monomorphism $\phi:{\mathbf{G}}\longrightarrow{}$Homeo$^+(V_g)$, where Homeo$^+(V_g)$ denotes the group of orientation-preserving homeomorphisms of $V_g$. Two actions $\phi_1$ and $\phi_2$ on $V_g$ are said to be equivalent if and only if there exists an orientation-preserving homeomorphism $h$ of $V_g$ such that $\phi_2(x)=h\circ\phi_1(x)\circ{h^{-1}}$ for all $x\in\mathbf{G}$. From \cite{MMZ}, the action of any finite group $\mathbf{G}$ on $V_g$ corresponds to a collection of graphs of groups.  We may assume these particular graphs of groups are in canonical form and satisfy a set of normalized conditions, which can be found in \cite{KM}.  

Let v\ $=(r,s,t,m,n)$ be an ordered $5$-tuple of nonnegative integers. The graph of groups $(\Gamma($v$),\mathbf{G(v)})$ in canonical form, shown in Figure 1, determines a handlebody orbifold $V(\Gamma($v$),\mathbf{G(v)})$. The orbifold $V(\Gamma($v$),{\mathbf{G(v)}})$ is constructed in a similar manner as described in \cite{KM}. Note that the quotient of any $\mathbb{Z}_{4}$-action on $V_g$ is an orbifold of this type, up to homeomorphism. 

\begin{figure}[hbtp]
\begin{center}
\begin{tikzpicture} 
\draw (3.75,2.65) node[anchor=north west, font=\footnotesize] {$\mathbb{Z}_2$};
\draw (2.75,2.65) node[anchor=north west, font=\footnotesize] {$\mathbb{Z}_2$};
\draw (-2,2.65) node[anchor=north west, font=\footnotesize] {$\mathbb{Z}_{4}$};
\draw (-0.85,2.65) node[anchor=north west, font=\footnotesize] {$\mathbb{Z}_{4}$};
\draw (1.95,2.65) node[anchor=north west, font=\footnotesize] {$\mathbb{Z}_2$};
\draw (1.5,3.45) node[anchor=north west, font=\footnotesize] {$\mathbb{Z}_2$};
\draw (0.5,3.45) node[anchor=north west, font=\footnotesize] {$\mathbb{Z}_2$};
\draw (-0.05,2.65) node[anchor=north west, font=\footnotesize] {$\mathbb{Z}_2$};
\draw (-4.95,2.65) node[anchor=north west, font=\footnotesize] {$\mathbb{Z}_{4}$};
\draw (-4.35,3.45) node[anchor=north west, font=\footnotesize] {$\mathbb{Z}_{4}$};
\draw (-2.8,2.65) node[anchor=north west, font=\footnotesize] {$\mathbb{Z}_{4}$};
\draw (-3.35,3.45) node[anchor=north west, font=\footnotesize] {$\mathbb{Z}_{4}$};
\draw (-0.25,-0.65) node[anchor=north west, font=\footnotesize] {$r$};
\draw (-3.75,2.5) node[anchor=north west, font=\footnotesize] {$s$};
\draw (-1.25,2.55) node[anchor=north west, font=\footnotesize] {$t$};
\draw (0.95,2.5) node[anchor=north west, font=\footnotesize] {$m$};
\draw (3.35,2.5) node[anchor=north west, font=\footnotesize] {$n$};
\draw (0,0) --(.75,2);
\draw (0,0) --(1.75,2);
\draw (0,0) --(-.5,2);
\draw (0,0) --(-3,2);
\draw (0,0) --(4,2);
\draw (0,0) --(3,2); 
\draw (0,0) --(-1.75,2);
\draw (0,0) --(-4,2);
\draw[rotate=45] (0,-.45) 
ellipse (4pt and 12pt);
\draw[rotate=-45] (0,-.45) 
ellipse (4pt and 12pt);
\draw (-4,2.45) ellipse (6pt and 12pt);
\draw (-3,2.45) ellipse (6pt and 12pt);
\draw (.75,2.45) ellipse (6pt and 12pt);
\draw (1.75,2.45) ellipse (6pt and 12pt);
\draw[fill] (3,2) circle (.05);
\draw[fill] (-4,2) circle (.05);
\draw[fill] (-3,2) circle (.05);
\draw[fill] (4,2) circle (.05);
\draw[fill] (-1.75,2) circle (.05);
\draw[fill] (0,0) circle (.05);
\draw[fill] (-.5,2) circle (.05);
\draw[fill] (1.75,2) circle (.05);
\draw[fill] (.75,2) circle (.05);
\draw[fill] (-0.2,-0.55) circle (.015);
\draw[fill] (0.2,-0.55) circle (.015);
\draw[fill] (0,-0.55) circle (.015);
\draw[fill] (-3.5,2) circle (.015);
\draw[fill] (-3.7,2) circle (.015);
\draw[fill] (-3.3,2) circle (.015);
\draw[fill] (1.2,2) circle (.015);
\draw[fill] (1.4,2) circle (.015);
\draw[fill] (1,2) circle (.015);
\draw[fill] (3.7,2) circle (.015);
\draw[fill] (3.3,2) circle (.015);
\draw[fill] (3.5,2) circle (.015);
\draw[fill] (-1.1,2) circle (.015);
\draw[fill] (-1.3,2) circle (.015);
\draw[fill] (-0.9,2) circle (.015);
\end{tikzpicture} 
\caption{($\Gamma($v$),{\mathbf{G(v)}})$}
\end{center}
\end{figure}


An explicit combinatorial enumeration of orientation-preserving $\mathbb{Z}_{p}$-actions on $V_g$, up to equivalence, is given in \cite{KM}. In this work we will be interested in examining the orientation-preserving geometric group actions on $V_g$ for the group $\mathbb{Z}_{4}$. The case for $\mathbb{Z}_{p^2}$, when $p$ is an odd prime is considered in \cite{JPL} and gives a different result. As we will see, there is exactly one equivalence class of $Z_4$-actions on the handlebody of genus 2. This result coincides with \cite{JK}. In this paper we will prove the following main theorem:



\begin{theorem}
If $\mathbb{Z}_4$ acts on $V_g$, where $g>0$, then $V_g/\mathbb{Z}_4$ is homeomorphic to $V(\Gamma({\mathbf{v}}),{\mathbf{G(v)}})$ for some 5-tuple ${\mathbf{v}}=(r,s,t,m,n)$ of nonnegative integers with $r+s+t+m+n>0$ and $g+3=4(r+s+m)+3t+2n$. The number of equivalence classes of $\mathbb{Z}_4$-actions on $V_g$ with this quotient type is $m$ if $r+s+t=0$, and $m+1$ if $r+s+t>0$.
\end{theorem}

To illustrate the theorem, let $g=3$. Then the genus equation becomes $6=4(r+s+m)+3t+2n$ so that $r+s+m$ must equal 0 or 1, and $(r,s,t,m,n)$ is one of $(0,0,2,0,0)$, $(1,0,0,0,1)$, $(0,1,0,0,1)$, or $(0,0,0,1,1)$. Applying Theorem 1.1 to these four possibilities shows that there are a total of $1+1+1+1=4$ equivalence classes of orientation-preserving $\mathbb{Z}_4$-actions on $V_3$. Some results that follow directly from Theorem 1.1:

\begin{corollary}
Every $\mathbb{Z}_4$-action on a handlebody of even genus must have an interval of fixed points and at least two fixed points on the boundary of the handlebody.
\end{corollary}

\begin{corollary}
Every $\mathbb{Z}_4$-action that is free on the boundary of the handlebody will have $t=n=0$ and $g\equiv1$(mod 4).
\end{corollary}

\section{\bf The Main Theorem}

The orbifold fundamental group of $V(\Gamma($v$), {\mathbf{G(v)}})$ is an extension of $\pi_1(V_g)$ by the group $\mathbf{G}$. We may view the fundamental group  as a free product $G_1*G_2*G_3*\cdots*{G_{r+s+t+m+n}}$, where $G_i$ is isomorphic to either $\mathbb{Z}$, $\mathbb{Z}_{4}\times\mathbb{Z}$, $\mathbb{Z}_{4}$, $\mathbb{Z}_2\times\mathbb{Z}$, or $\mathbb{Z}_2$. We establish notation similar to \cite{KM} and denote the generators of the orbifold fundamental group by $\{a_i: 1\leq{i}\leq{r}\}\cup\{b_j,c_j:1\leq{j}\leq{s}\}\cup\{d_k:1\leq{k}\leq{t}\}\cup\{e_l,f_l:1\leq{l}\leq{m}\}\cup\{g_q:1\leq{q}\leq{n}\}$ such that $b_j^{4}=d_k^{4}=1$, $[b_j,c_j]=1$, $e_l^2=g_q^2=1$, and $[e_l,f_l]=1$. 

Consider the set of pairs $((\Gamma({\mathbf{v}}),{\mathbf{G(v)}}),\lambda)$, where $\lambda$ is a finite injective epimorphism from $\pi_1^{orb}(V(\Gamma({\mathbf{v}}),{\mathbf{G(v)}}))$ onto $\mathbb{Z}_{4}$. We say $\lambda$ is finite injective since the kernel of $\lambda$ is a free group of rank $g$. We consider only finite injective epimorphisms such that ker($\lambda$)$=$im($\nu_*$) for some orbifold covering $\nu:V\longrightarrow V(\Gamma({\mathbf{v}}),{\mathbf{G(v)}})$. Since $V$ is a handlebody with torsion free fundamental group, $V$ is homeomorphic to a handlebody $V_g$ of genus $g=1-4\chi(\Gamma({\mathbf{v}}),{\mathbf{G(v)}})$. Define an equivalence relation on this set of pairs by setting $((\Gamma({\mathbf{v}}),{\mathbf{G(v)}}),\lambda)\equiv(({\Gamma}({\mathbf{v}}),{\mathbf{G(v)}}),\lambda')$ if and only if there exists an orbifold homeomorphism $h:V(\Gamma({\mathbf{v}}),{\mathbf{G(v)}})\longrightarrow{V(\Gamma({\mathbf{v}}),{\mathbf{G(v)}})}$ such that $\lambda'=\lambda\circ{h_*}$. We define the set $\Delta(\mathbb{Z}_{4},V_g,V(\Gamma({\mathbf{v}}),{\mathbf{G(v)}}))$ to be the set of equivalence classes $[((\Gamma({\mathbf{v}}),{\mathbf{G(v)}}),\lambda)]$ under this relation.

Denote the  set set of equivalence classes $\mathscr{E}(\mathbb{Z}_{4},V_g,V(\Gamma({\mathbf{v}}),{\mathbf{G(v)}}))$ to be the set  $\{[\phi]\ |\ \phi:\mathbb{Z}_{4}\longrightarrow{Homeo^+(V_g)}\ and\ V_g/\phi\simeq{V(\Gamma({\mathbf{v}}),{\mathbf{G(v)}})}\}$. Note that given any $\mathbb{Z}_{4}$-action $\phi:\mathbb{Z}_{4}\longrightarrow{Homeo^+(V_g)}$, it must be the case that for some $V(\Gamma($v$),{\mathbf{G(v)}})$, $[\phi]\in\mathscr{E}(\mathbb{Z}_{4},V_g,V(\Gamma({\mathbf{v}}),{\mathbf{G(v)}}))$. The following proposition has a similar proof technique as found in \cite{KM}.

\begin{proposition}
Let ${\mathbf{v}}=(r,s,t,m,n)$. The set $\mathscr{E}(\mathbb{Z}_{4},V_g,V(\Gamma({\mathbf{v}}))$ is in one-to-one correspondence with the set $\Delta(\mathbb{Z}_{4},V_g,V(\Gamma({\mathbf{v}}),{\mathbf{G(v)}}))$ for every $g>0$.
\end{proposition}

To prove the main theorem, we count the number of elements in the delta set and use the one-to-one correspondence given in Proposition 2.1 to give the total count for the set $\mathscr{E}(\mathbb{Z}_{4},V_g,V(\Gamma({\mathbf{v}}),{\mathbf{G(v)}}))$. We resort to the following lemma to help count the number of elements in the delta set. The proof is an adaptation from \cite{KM}.

\begin{lemma}
\label{real}
If $\alpha$ is an automorphism of $\pi_1^{orb}(V(\Gamma({\mathbf{v}}),\mathbf{G(v)}))$, then $\alpha$ is realizable $[\alpha=h_*$ for some orientation-preserving homeomorphism $h:V(\Gamma({\mathbf{v}}),{\mathbf{G(v)}})\longrightarrow{V(\Gamma({\mathbf{v}}),{\mathbf{G(v)}})}]$ if and only if 
\begin{align}
\alpha(b_j)&=x_jb_{\sigma(j)}^{\varepsilon_j}x_j^{-1},\notag \\
\alpha(c_j)&=x_jb_{\sigma(j)}^{{v_j}}c_{\sigma(j)}^{\varepsilon_j}x_j^{-1},\notag \\
\alpha(d_k)&=y_kd_{\tau(k)}^{\delta_k}y_k^{-1},\notag \\
\alpha(e_l)&=u_le_{\gamma(l)}^{\varepsilon_l'}u_l^{-1},\notag \\
\alpha(f_l)&=u_le_{\gamma(l)}^{{w_l}}f_{\gamma(l)}^{\varepsilon_l'}u_l^{-1}, and\notag \\
\alpha(g_q)&=z_qg_{\xi(q)}^{\delta_q'}z_q^{-1},\notag
\end{align}
for some $x_j, y_k, u_l, z_q\in{\pi_1^{orb}(V(\Gamma({\mathbf{v}}),\mathbf{G(v)}))}$; $\sigma\in{\sum_s}$, $\tau\in{\sum_t}$, $\gamma\in{\sum_m}$, $\xi\in{\sum_n}$; $\varepsilon_j, \delta_k, \varepsilon_l', \delta_q'\in{\{+1,-1\}}$; and $0\leq{v_j}<4$, $0\leq{w_l}<2$.\\
Note that $\Sigma_l$ is the permutation group on l letters.
\end{lemma}

Note that from \cite{FR}, a generating set for the automorphisms of the orbifold fundamental group $\pi_1^{orb}(V(\Gamma($v$),\mathbf{G(v)}))$ is the set of mappings $\{\rho_{ji}(x), \lambda_{ji}(x), \mu_{ji}(x), \omega_{ij}, \sigma_i, \phi_i\}$ whose definitions may be found in \cite{FR}. The first five maps are realizable. The realizable $\phi_i$'s are of the form found in Lemma 2.2 and will be used in the remaining arguments of this paper.

\begin{lemma}
Let ${\mathbf{v}}=(r,s,t,m,n)$ with $m>0$ and let \\$\lambda_1,\lambda_2:\pi_1^{orb}(V(\Gamma({\mathbf{v}}),{\mathbf{G(v)}}))\longrightarrow\mathbb{Z}_{4}$ be two finite injective epimorphisms such that there exists a $j$ with $\lambda_1(f_j)$ being a generator of $\mathbb{Z}_{4}$ and $\lambda_2(f_i)$ is not a generator of $\mathbb{Z}_{4}$ for all $i$. Then $\lambda_1$ and $\lambda_2$ are not equivalent.
\end{lemma}

\begin{proof}
Let $\lambda_1,\lambda_2:\pi_1^{orb}(V(\Gamma($v$),{\mathbf{G(v)}}))\longrightarrow\mathbb{Z}_{4}$ be two finite injective epimorphisms such that $\lambda_1$ sends $f_j$ to a generator of $\mathbb{Z}_{4}$ for some $j$ and $\lambda_2$ does not send $f_i$ to a generator of $\mathbb{Z}_{4}$ for all $i$. We may assume that ${\lambda_2}(f_i)=0$ for all $i$ by composing $\lambda_2$ with the realizable automorphism $\prod\phi_i$, where $\phi_i$ sends the generator $f_i$ to the element $e_i^{w_i}f_i$ and leaves all other generators fixed. Note that $w_i=0$ if $\lambda_2(f_i)=0$ and $w_i=1$ if $\lambda_2(f_i)=2$. To show that $\lambda_1$ and ${\lambda_2}$ are not equivalent we will consider the element $f_j$ such that $\lambda_1(f_j)$ generates $\mathbb{Z}_{4}$. For contradiction, assume that $\lambda_1$ is equivalent to ${\lambda_2}$. Then by Lemma \ref{real}, there exists a realizable automorphism $\alpha$ such that $\alpha(f_j)=ue_m^wf_m^{\pm1}u^{-1}$, where $u\in{\pi_1^{orb}(V(\Gamma(}$v${),{\mathbf{G(v)}}))}$ and $0\leq{w}<{2}$. Hence $\lambda_1(f_j)=w{\lambda_2}(e_m)$, where ${\lambda_2}(e_m)$ is a multiple of $2$, and hence $\lambda_1(f_j)$ is a multiple of $2$. This is impossible since $\lambda_1(f_j)$ is a generator of $\mathbb{Z}_{4}$. Therefore $\lambda_1$ and $\lambda_2$ cannot be equivalent, proving the lemma.
\end{proof}

\begin{lemma}
Let ${\mathbf{v}}=(r,s,t,m,n)$ and let $\lambda:\pi_1^{orb}(V(\Gamma({\mathbf{v}}), {\mathbf{G(v)}}))\longrightarrow{\mathbb{Z}_{4}}$ be a finite injective epimorphism. There exists a finite injective epimorphism $\tilde{\lambda}:\pi_1^{orb}(V(\Gamma({\mathbf{v}}), {\mathbf{G(v)}}))\longrightarrow{\mathbb{Z}_{4}}$ equivalent to $\lambda$ such that the following hold:
\begin{list}{(\arabic{qcounter})~}{\usecounter{qcounter}}
\item $\tilde{\lambda}(a_{1})=\cdots=\tilde{\lambda}(a_{r})=1$.
\item $\tilde{\lambda}(b_{1})=\cdots=\tilde{\lambda}(b_{s})=1$.
\item $\tilde{\lambda}(c_i)=0$ for all $1\leq{i}\leq{s}$.
\item $\tilde{\lambda}(d_{1})=\cdots=\tilde{\lambda}(d_{t})=1$.
\item $\tilde{\lambda}(e_{1})=\cdots=\tilde{\lambda}(e_{m})={2}$.
\item $\tilde{\lambda}(f_i)=1$ for all $i\leq{k}$ some $0\leq{k}\leq{m}$.
\item $\tilde{\lambda}(f_i)=0$ for all $k<{i}\leq{m}$.
\item $\tilde{\lambda}(g_{1})=\cdots=\tilde{\lambda}(d_{n})=2$.
\end{list}
\end{lemma}

\begin{proof}
Let $\lambda:\pi_1^{orb}(V(\Gamma($v$), {\mathbf{G(v)}}))\longrightarrow{\mathbb{Z}_{4}}$ be a finite injective epimorphism. Properties (5) and (8) must occur since $\lambda$ is finite injective. Property (4) follows by composing $\lambda$ with the realizable automorphism \\$\prod\phi_i$, where $\phi_i$ sends the generator $d_i$ to the element $d_i^{{\varepsilon}_i}$ and leaves all other generators fixed. Note that $\varepsilon_i=1$ if $\lambda(d_i)=1$ and $\varepsilon_i=-1$ if $\lambda(d_i)=3$. Property (2) follows by a similar technique. Assuming property (2) holds, property (3) follows by composing $\lambda$ with the realizable automorphism $\prod\phi_i$, where $\phi_i$ sends the generator $c_i$ to the element $b_i^{-\lambda(c_i)}c_i$ and leaves all other generators fixed. To show properties (6) and (7) hold, we may compose $\lambda$ with the realizable automorphism $\prod\phi_i$, where $\phi$ sends the generator $f_i$ to the element $e_i^{z_i}f_i$ and leaves all other generators fixed. Note that $z_i=1$ if $\lambda(f_i)=2$, $z_i=2$ if $\lambda(f_i)=0$ or $\lambda(f_i)=1$, and $z_i=-1$ if $\lambda(f_i)=3$. Furthermore, composing $\lambda$ with the realizable automorphisms $\omega_{ij}$ we may interchange $f_i$ as needed so that the first $k$ generators map to 1 and the last $m-k$ generators map to 0. Finally, to prove property (1) we may assume that there exists an element $x\in G_j$ (where $G_j$ is either $\mathbb{Z}$, $\mathbb{Z}_4$, $\mathbb{Z}_4\times\mathbb{Z}$, or $\mathbb{Z}_2\times\mathbb{Z}$) such that $\lambda(x)=1$. Note that we may compose $\lambda$ with a realizable automorphism that sends $x$ to $x^{-1}$ if needed. Now compose $\lambda$ with the realizable automorphism $\prod\rho_{ji}(x^{-\lambda(a_i)+1})$. It may be shown that $(\lambda\circ\alpha)(a_i)=1$ for all $i$.
\end{proof}

\begin{proposition}
Let ${\mathbf{v}}=(r,s,t,m,n)$ with $m>0$ and let \\$\lambda,\lambda':\pi_1^{orb}(V(\Gamma({\mathbf{v}}),{\mathbf{G(v)}}))\longrightarrow\mathbb{Z}_{4}$ be two finite injective epimorphisms that satisfy the conclusion of Lemma 2.4, where $\lambda(f_i)=1$ for all $1\leq{i}\leq{k}$ and $\lambda'(f_i)=1$ for all $1\leq{i}\leq{k'}$. Then $\lambda$ is equivalent to $\lambda'$ if and only if $k=k'$. 
\end{proposition}

\begin{proof}
For a contradiction, assume that $\lambda$ is equivalent to $\lambda'$ and $k\neq{k'}$. Without loss of generality we may assume that $k>k'$. Hence, $\lambda$ maps at least one more generator $f_i$ to 1 as does $\lambda'$. This would mean that there must exist a realizable automorphism $\alpha$ such that $(\lambda\circ\alpha)(f_{k'+1})=0$. By Lemma 2.2, this is impossible. Thus, $k=k'$. For the reverse implication suppose that $k=k'$. Then $\lambda=\lambda'$, proving the proposition.
\end{proof}

We will now prove the main theorem.

\begin{proof}[Proof of Theorem 2.2]
Define $\Delta_0(\mathbb{Z}_{4},V_g,V(\Gamma({\mathbf{v}}),{\mathbf{G(v)}}))$ to be the set of equivalence classes $[(\Gamma({\mathbf{v}}),{\mathbf{G(v)}}),\lambda]$ such that $\lambda(f_i)=0$ for all $1\leq{i}\leq{m}$. Define $\Delta_1(\mathbb{Z}_{4},V_g,V(\Gamma({\mathbf{v}}),{\mathbf{G(v)}}))$ to be the set of equivalence classes $[(\Gamma({\mathbf{v}}),{\mathbf{G(v)}}),\lambda]$ such that $\lambda(f_i)=1$ for at least one $i$ such that $1\leq{i}\leq{m}$. By Lemma 2.3, the delta set $\Delta(\mathbb{Z}_{4},V_g,V(\Gamma({\mathbf{v}}),{\mathbf{G(v)}}))$ is the disjoint union $\Delta_0(\mathbb{Z}_{4},V_g,V(\Gamma({\mathbf{v}}),{\mathbf{G(v)}}))\dot{\bigcup}\Delta_1(\mathbb{Z}_{4},V_g,V(\Gamma({\mathbf{v}}),{\mathbf{G(v)}}))$. Hence, the order of the delta set is the sum of the orders of the two sets $\Delta_0(\mathbb{Z}_{4},V_g,V(\Gamma({\mathbf{v}}),{\mathbf{G(v)}}))$ and $\Delta_1(\mathbb{Z}_{4},V_g,V(\Gamma({\mathbf{v}}),{\mathbf{G(v)}}))$. Applying Lemma 2.4 and Proposition 2.5, $|\Delta_1(\mathbb{Z}_{4},V_g,V(\Gamma({\mathbf{v}}),{\mathbf{G(v)}}))|=m$ and \\$|\Delta_0(\mathbb{Z}_{4},V_g,V(\Gamma({\mathbf{v}}),{\mathbf{G(v)}}))|=1$. Hence by Proposition 2.1, the theorem follows.
\end{proof}

\bibliographystyle{plain}

\begin{thebibliography}{10}

\smallskip

\bibitem{FR}
D. I. Fuchs-Rabinovitch,
\emph{On the automorphism group of free products},
I, Mat. Sb. 8 (1940), 265-276.

\smallskip

\bibitem{KM}
J. Kalliongis and A. Miller,
\emph{Equivalence and strong equivalence of actions on handlebodies},
Trans. Amer. Math. Soc. 308 (1988).

\smallskip

\bibitem{JK}
J. Kim,
\emph{Structures of geometric quotient orbifolds of three-dimensional $g$-manifolds of genus 2},
J. Korean Math. Soc. 46 (2009), No. 4, pp. 859-893.

\smallskip


\bibitem{MMZ}
D. McCullough, A. Miller, and B. Zimmerman,
\emph{Group actions on handlebodies}, Proc. London Math. Soc. (3) 59 (1989) 373-416.



\smallskip


\bibitem{JPL}
J. Prince-Lubawy,
\emph{Equivalence of cyclic $p^2$-actions on handlebodies}. (In preparation)

\smallskip

\end{thebibliography}

\end{document}